\documentclass[11pt]{article}
\usepackage{geometry} 
\usepackage{amsmath,amsthm,amssymb}                
\usepackage{graphicx, color}
\usepackage{amssymb}
\usepackage{epstopdf}
\usepackage{pdfsync}

\newtheorem{definition}{Definition}[section]
\newtheorem{lemma}[definition]{Lemma}
\newtheorem{theorem}[definition]{Theorem}
\newtheorem{proposition}[definition]{Proposition}
\newtheorem{corollary}[definition]{Corollary}
\newtheorem{remark}[definition]{Remark}
\newtheorem{example}[definition]{Example}

\numberwithin{equation}{section}

\def\e{\varepsilon}

\usepackage{lipsum}

\newcommand\blfootnote[1]{%
  \begingroup
  \renewcommand\thefootnote{}\footnote{#1}%
  \addtocounter{footnote}{-1}%
  \endgroup
}

\usepackage{bbm,amsmath,amsthm,epsfig,latexsym,marvosym,mathrsfs}

\title{Continuity of some non-local functionals with respect to \\ a convergence of the underlying measures}
\author{Andrea Braides and Gianni Dal Maso \\ 
SISSA, via Bonomea 265, Trieste, Italy}
\date{}
\begin{document}

\maketitle

\begin{abstract}
We study some non-local functionals on the Sobolev space $W^{1,p}_0(\Omega)$ involving a double integral on $\Omega\times\Omega$ with respect to a measure $\mu$. We introduce a suitable notion of convergence of measures on product spaces which implies a stability property in the sense of $\Gamma$-convergence of the corresponding functionals.
    
    
{\bf AMS Class:} 49J45, 46E35, 28A33, 28A35
\end{abstract}

\blfootnote{Preprint SISSA 8/2022/MATE}

\section{Introduction}
In this paper we give a contribution to the study of $\Gamma$-limits of non-local integral functional, for which only few results are available in the literature (see for instance \cite{Mosco}).
We consider sequences of  integrals of the type 
\begin{equation}\label{effe-k}
\int_{\Omega\times\Omega} f(u(x),u(y))  d\mu_k(x,y) +\int_\Omega g_k(x,\nabla u(x))\,dx,
\end{equation}
where $\Omega$ is a bounded open subset of $\mathbb R^d$, with $d\ge1$.
These functionals have a non-local term 
$$ 
F_k(u):=\int_{\Omega\times\Omega} f(u(x),u(y))  d\mu_k(x,y)
$$
depending on a fixed function $f\colon \mathbb R\times\mathbb R\to[0,+\infty)$ and varying  positive bounded measures $\mu_k$ on $\Omega\times\Omega$, while the local term 
$$ 
G_k(u):=\int_\Omega g_k(x,\nabla u(x))\,dx
$$
depends on a function $g_k\colon \Omega\times\mathbb R^d\to[0,+\infty)$. These functionals are
defined  for $u$ in the Sobolev space  $W^{1,p}_0(\Omega)$.

 We assume that the functions $g_k$ satisfy usual growth conditions and that the integral functionals 
 $G_k$ $\Gamma$-converge in  the weak topology in $W^{1,p}_0(\Omega)$ to a functional $G$ of the same form, with integrand $g$. 
 
We address the question of the stability for functionals in \eqref{effe-k}; more precisely, we focus on a notion of convergence on $\mu_k$ such that the functionals $F_k+G_k$ $\Gamma$-convergence with respect to the weak topology in $W^{1,p}_0(\Omega)$  to a functional of the form 
\begin{equation}\label{intro:1}
	\int_{\Omega\times\Omega} f(u(x),u(y))  d\mu(x,y) + 
	\int_\Omega g(x,\nabla u(x))\,dx
\end{equation}
for a limit measure $\mu$.
 Under some additional assumptions, we also obtain the convergence of $F_k+G_k$ in the sense of Mosco convergence in $W^{1,p}_0(\Omega)$. 
 
 If $p=2$, $g_k(x,\cdot)$ are quadratic forms, and $f(s,t)=|t-s|^2$, then the study of such functionals can be framed within the theory of Dirichlet Forms \cite{Fukushima}, where the Beurling-Deny formula ensures, under suitable assumptions, that the $\Gamma$-limit of $F_k+G_k$ can be represented analogously (see \cite{Mosco}). 
The extension of that theory is not immediate in a non-quadratic setting or when $f$ is an arbitrary continuous function.

Note that, under suitable growth conditions on $f$, stability is easily proved under the strong assumption of convergence of $\mu_k$ to 
$\mu$ in the space $W^{-1,q}(\Omega\times\Omega)$ dual to $W^{1,p}_0(\Omega\times\Omega)$. However, this result is not satisfactory, since such a space may fail to contain relevant measures $\mu$, depending on the value of $p$, such as Dirac deltas if $p<2d$. 

We introduce a wider space of measures  on $\Omega\times \Omega$, together with a new notion of norm, inspired by a convergence that is used in the theory of {\em graphons} \cite{Bo,Lo}. The latter can be be seen as limits of Dirichlet forms on dense graphs (for an interpretation in terms of $\Gamma$-convergence we refer to \cite{BCD}).

We prove that,  if $\mu_k$ are non-negative measures that converge to $\mu$ with respect to that `graphon' norm and $\mu_k(\Omega\times\Omega)\to \mu(\Omega\times\Omega)$, then the functionals  defined by \eqref{effe-k} $\Gamma$-converge under the only assumption that $f$ be continuous and a very mild technical assumption (see \eqref{trundec}).

\smallskip
We now describe more in detail the content of the paper.
In Section \ref{prelSobolev} we recall some preliminaries on Sobolev functions and introduce the quasicontinuous representative $\widetilde u$ of a function $u\in W^{1,p}_0(\Omega)$, which is needed in the precise definition of the functionals  \eqref{effe-k} and \eqref{intro:1}, when $\mu_k$ or $\mu$ are not absolutely continuous with respect to the Lebesgue measure.

In Section \ref{Sobolev graphons} we introduce the space ${\mathcal M}^{1,p}(\Omega\times\Omega)$ of Radon measures on $\Omega\times\Omega$ with finite `Sobolev cut norm',
defined as
\begin{eqnarray*}
\|\mu\|_\square:=\sup\Bigl\{\Bigl|\int_{\Omega\times\Omega}\varphi(x)\psi(y)d\mu(x,y)\Bigr|: \varphi,\psi\in C^\infty_c(\Omega),
\|\varphi\|_{1,p}\,,\|\psi\|_{1,p}\le 1\Bigr\},
\end{eqnarray*}
where $\|\cdot\|_{1,p}$ is the norm in $W^{1,p}_0(\Omega)$. We prove that, if $\mu\in{\mathcal M}^{1,p}(\Omega\times\Omega)$ and $u, v\in W^{1,p}_0(\Omega)\cap L^\infty(\Omega)$ with compact support in $\Omega$, we have
$$
\Bigl|\int_{\Omega\times\Omega}\widetilde u(x)\widetilde v(y)d\mu(x,y)\Bigr|\le \|\mu\|_\square \|u\|_{1,p}\|v\|_{1,p} ,
$$
and the integral does not depend on the choice of the quasicontinuous representatives (see Theorem \ref{finitemu}).

In Section \ref{continconv} we prove some continuity results for double integrals. We consider
$\mu_k, \mu\in\mathcal M^{1,p}_+(\Omega\times\Omega)$ with  $\mu_k(\Omega{\times}\Omega)<+\infty$,  $\mu(\Omega{\times}\Omega)<+\infty$, and 
$$
\|\mu_k-\mu\|_\square\to 0\quad\hbox{  and }\quad\mu_k(\Omega{\times}\Omega)\to \mu(\Omega{\times}\Omega).
$$
If $u_k, v_k$ are sequences in $W^{1,p}_0(\Omega)\cap L^\infty(\Omega)$ converging to $u,v$  weakly in $W^{1,p}_0(\Omega)$ and with $\sup_k(\|u_k\|_\infty+\|v_k\|_\infty)<+\infty$, then we have
$$
\lim_{k\to+\infty} \int_{\Omega\times\Omega} f(\widetilde u_k(x),\widetilde v_k(y)) d\mu_k(x,y)=
\int_{\Omega\times\Omega}f(\widetilde u(x),\widetilde v(y)) d\mu(x,y) 
$$
for all continuous functions $f\colon\mathbb R^2\to \mathbb R$ (see Corollary \ref{cortre}). This result is obtained by considering first the special case $f(s,t)=st$ (see Corollary \ref{corto}), from which the result can be obtained for a general polynomial (see Proposition \ref{contpol}), and eventually for an arbitrary continuous function by approximation.

Finally, in Section \ref{gammaconv} we consider non-negative continuous functions $f$, and prove the above-mentioned $\Gamma$-convergence results (see Theorem \ref{teogamma}). Note that no convexity assumption on $f$ is needed.



\section{Preliminaries on fine properties of Sobolev functions}\label{prelSobolev}

Throughout the paper we fix $1<p<+\infty$,  $d\ge1$, and a bounded open subset $\Omega$ of $\mathbb R^d$. We consider the Sobolev space $W^{1,p}_0(\Omega)$ endowed with the norm $$\|u\|_{1,p}=\Bigl(\int_\Omega|\nabla u|^p dx\Bigr)^{1/ p}. $$
Its dual is denoted by $W^{-1,q}(\Omega)$, where $\frac1p+\frac1q=1$, and this norm is denoted by $\|\cdot\|_{-1,q}$.

For all subset $A\subset\Omega$ the {\em capacity} of  in $\Omega$ is defined as
$$
C_{1,p}(A):=\inf\Bigl\{\int_\Omega |\nabla u|^p\,dx: u\in W^{1,p}_0(\Omega),\  u\ge 1 \hbox{ a.e.~in a neighbourhood of } A\Bigl\}.
$$
A function $f\colon \Omega\to\mathbb R$ is {\em quasicontinuous} if for every $\varepsilon>0$ there exists an open set $U\subset\Omega$ with $C_{1,p}(U)<\varepsilon$ such that the restriction of to $\Omega\setminus U$ is continuous. It is known (see e.g.~\cite{finlandesi,EG}) that each function $u\in W^{1,p}_0(\Omega)$ has a {\em Borel  quasicontinuous representative}, which we denote by $\widetilde u$, in the sense that $\widetilde u$ is Borel measurable and quasicontinuous and 
$\widetilde u=u$ almost everywhere in $\Omega$. Such a representative is unique up to sets of zero capacity. 

Let ${\mathcal M}(\Omega)$ denote the space of signed Radon measures on $\Omega$, which can be identified with the dual of $C^0_c(\Omega)$. 
We say that $\mu\in {\mathcal M}(\Omega)$ belongs to $W^{-1,q}(\Omega)$ if there existe $C\ge0$ such that
$$
\Bigl|\int_\Omega \varphi\,d\mu\Bigr|\le C\|\varphi\|_{1,p}\qquad\hbox{ for all }\varphi\in C^{\infty}_c(\Omega).
$$ In this case there exists a unique $T_\mu\in W^{-1,q}(\Omega)$ such that
\begin{equation}\label{timu}
\langle T_\mu,\varphi\rangle = \int_\Omega \varphi d\mu
\qquad\hbox{ for all }\varphi\in C^{\infty}_c(\Omega),\end{equation}
 where $\langle \cdot,\cdot\rangle$ denotes the duality product between 
$W^{-1,q}(\Omega)$ and $W^{1,p}_0(\Omega)$. In other words $\mu$ and $T_\mu$ coincide as distributions on $\Omega$. For notational convenience, we shall sometimes directly write $\mu$ in the place of $T_\nu$ when the distinction between the two is not relevant.

In the following theorem we recall a property of sets of zero  $C_{1,p}$-capacity (see \cite{Grun-Rehomme}) and an integral representation of $T_\mu$ that can be deduced from a result of Brezis and Browder \cite{BB}.

In order to simplify the notation, we introduce the space
$$
W^{1,p}_c(\Omega):=\{ u\in W^{1,p}_0(\Omega): u \hbox{ has compact support in } \Omega\}.
$$

\begin{theorem}\label{BreBro}
Let $\mu\in {\mathcal M}(\Omega)\cap W^{-1,q}(\Omega)$. Then for every $A\subset\Omega$ with $C_{1,p}(A)=0$-capacity  $A$ is $|\mu|$-measurable and 
\begin{equation}\label{zerocapBB}
|\mu|(A)=0
\end{equation}
If $u\in W^{1,p}_c(\Omega)\cap L^\infty(\Omega)$, then 
\begin{equation}\label{BBfor}
\langle T_\mu,u\rangle = \int_\Omega \widetilde u d\mu .
\end{equation}
If, in addition, $\mu\ge 0$, then,  for all $u\in W^{1,p}_0(\Omega)$, we have $\widetilde u\in L^1(\Omega;\mu)$ and \eqref{BBfor} holds.
\end{theorem}

\section{Sobolev graphons}\label{Sobolev graphons}
{\em Graphons} are functions $\rho$ defined on $(0,1)\times (0,1)$ introduced to study functionals of the form 
$$
\int_{(0,1)\times (0,1)} f(u(x),u(y))\rho(x,y)\,dx\,dy,
$$
defined for  $u\in L^\infty(0,1)$, especially when $f(u,v)=|u-v|^2$. 
The functionals above are introduced as a generalization, and in some sense a limit, of energies on dense graphs \cite{Bo,Lo}. 
To that end, the space of graphons is equipped with the so-called {\em cut norm }
\begin{eqnarray}\label{normagraphon0}
\|\rho\|_\square :=\sup_{\varphi,\psi\colon(0,1)\to[0,1]}\ \biggl|\int_{(0,1)\times (0,1)}\varphi(x)\psi(y)\rho(x,y)\,dx\,dy\biggr|.
\end{eqnarray}

We extend the definition of cut norm to arbitrary measures defined on $\Omega\times \Omega$, with $\Omega$ in the place of $(0,1)$. For our purpose it is convenient to take test functions in the Sobolev space $W^{1,p}_0(\Omega)$.

\begin{definition}
The space  ${\mathcal M}^{1,p}(\Omega\times\Omega)$ is defined as
\begin{eqnarray}
{\mathcal M}^{1,p}(\Omega\times\Omega):=\bigl\{\mu\in {\mathcal M}(\Omega\times\Omega):\|\mu\|_\square<+\infty\bigr\},
\end{eqnarray}
where the {\em Sobolev cut norm} of $\mu$ is defined as 
\begin{eqnarray}\label{normagraphon}
\|\mu\|_\square:=\sup\Bigl\{\Bigl|\int_{\Omega\times\Omega}\varphi(x)\psi(y)d\mu(x,y)\Bigr|: \varphi,\psi\in C^\infty_c(\Omega),
\|\varphi\|_{1,p}\,,\|\psi\|_{1,p}\le 1\Bigr\}.
\end{eqnarray}
We let ${\mathcal M}^{1,p}_+(\Omega\times\Omega)$ denote the cone of the positive measures in ${\mathcal M}^{1,p}(\Omega\times\Omega)$.
\end{definition}

\begin{remark}\label{rem3}\rm
We make some observations on the convergence in the space ${\mathcal M}^{1,p}(\Omega\times\Omega)$, and in particular we compare it with weak$^*$ convergence in ${\mathcal M}(\Omega\times\Omega)$ and with the convergence  $W^{-1,q}({\Omega{\times}\Omega})$.

(i) $\|\cdot\|_\square$ defines a norm on ${\mathcal M}^{1,p}(\Omega\times\Omega)$;

(ii) if $\mu$ belongs to $W^{-1,q}(\Omega\times\Omega)$, then $\|\mu\|_\square\le \|\mu\|_{-1,q}$. Indeed in such a case the function $\Phi(x,y)= \varphi(x)\psi(y)$ belongs to $W^{1,p}_0(\Omega\times\Omega)$ 
and 
\begin{eqnarray*}
\int_{\Omega\times \Omega}|\nabla \Phi|^p dx\,dy&=& \int_{\Omega\times \Omega}|\psi(y)\nabla \varphi(x)+ \varphi(x)\nabla\psi(y)|^p dx\,dy\\
&\le& C\int_{\Omega\times \Omega}|\psi(y)|^p|\nabla \varphi(x)|^p+ |\varphi(x)|^p|\nabla\psi(y)|^p dx\,dy
\le C\end{eqnarray*}
so that 
$$
\Bigl|\int_{\Omega\times\Omega}\varphi(x)\psi(y)d\mu(x,y)\Bigr|\le C\|\mu\|_{-1,q} ;
$$

(iii) if $\mu_k, \mu \in {\mathcal M}^{1,p}({\Omega\times\Omega})$ and $\|\mu_k-\mu\|_\square\to0$, then 
\begin{equation}\label{convint}
\int_{\Omega\times\Omega}\varphi(x)\psi(y)d\mu_k(x,y)\to \int_{\Omega\times\Omega}\varphi(x)\psi(y)d\mu(x,y)
\end{equation}
for all $\varphi,\psi\in C^\infty_c(\Omega)$. If in addition  $\sup_k|\mu_k|(\Omega{\times}\Omega)<+\infty$, then \eqref{convint} and a density argument imply that $\mu_k$ converges to $\mu$ weakly$^*$ in ${\mathcal M}(\Omega\times\Omega)$;

(iv) if $\mu^1_k$ and $\mu^2_k$ are such that $\mu^j_k$ converge to $\mu^j$ in $W^{-1,q}(\Omega)$, then 
the measures $\mu_k= \mu^1_k\otimes \mu^2_k $ converge to $\mu= \mu^1\otimes \mu^2 $;

(v) if $\sup_k \|\mu_k\|_\square <+\infty$ and $\mu_k$  converges to some $\mu$ weakly$^*$ in ${\mathcal M}({\Omega\times\Omega})$, then $\mu\in {\mathcal M}^{1,p}({\Omega\times\Omega})$ and $\|\mu\|_\square\le\liminf_k \|\mu_k\|_\square$;

%
\end{remark}

Note that ${\mathcal M}^{1,p}({\Omega{\times}\Omega})$ is strictly larger than ${\mathcal M}({\Omega{\times}\Omega})\cap W^{-1,q}({\Omega{\times}\Omega})$, and in particular its convergence is weaker than convergence in $W^{-1,q}({\Omega{\times}\Omega})$. Some examples of this inclusion are given below.

\begin{example} \rm (i) The first example is simply a Dirac delta $\mu= \delta(x_0, y_0)$ for $x_0,y_0\in\Omega$.
Indeed if $d<p<2d$, then $\mu$ does not belong to 
$W^{-1,q}(\Omega\times\Omega)$, while
$$
\Bigl|\int_{\Omega\times\Omega}\varphi(x)\psi(y)d\mu(x,y)\Bigr|= \Bigl|\varphi(x_0)\psi(y_0)\Bigr|\le C\|\varphi\|_{1,p}\|\psi\|_{1,p}
$$
where the inequality follows from the embedding of $W^{1,p}_0(\Omega)$ into $L^\infty(\Omega)$.

\smallskip
(ii) We can also exhibit an example where $\mu$ is absolutely continuous with respect to ${\cal L}^d$ with density of the form $m(x)m(y)$. We choose $d=1$, $p=2$, $\Omega=(-1,1)$, and 
$$
m(x)=\frac{1}{|x|\,\log|x|\,\log|\log x|\, \log^2|\log|\log x||}.
$$
Note that $m\in L^1(0,1)$ so that $\mu\in{\mathcal M}^{1,2}(\Omega\times\Omega)$. If we take
$$
w(x,y)= \begin{cases}\Bigl|\log\bigl|\log\sqrt{x^2+y^2}\bigr|\Bigr|  & \hbox{ if } \sqrt{x^2+y^2}<{1}/{e},\\
0 & \hbox{ otherwise,}\end{cases}
$$
then $w\in W^{1,2}_0((-1,1)^2)$ but $\int_{(-1,1)^2} w\,d\mu=+\infty$. By Theorem \ref{BreBro}, this shows that $\mu\not\in W^{-1,2}((-1,1)^2)$. 
\end{example}

\begin{theorem}\label{finitemu}
Let $\mu\in {\mathcal M}^{1,p}(\Omega\times\Omega)$. Then for every  $A\subset\Omega$ with $C_{1,p}(A)=0$ the sets $A\times\Omega$ and $\Omega\times A$ are $|\mu|$-measurable and 
\begin{equation}\label{muzero}
|\mu|(A\times\Omega)=|\mu|(\Omega\times A)=0 .
\end{equation}
Moreover, for all $u,v\in W^{1,p}_c(\Omega)\cap L^\infty(\Omega)$ we have 
\begin{equation}\label{intcapuv}
\Bigl|\int_{\Omega\times\Omega}\widetilde u(x)\widetilde v(y)d\mu(x,y)\Bigr|\le \|\mu\|_\square \|u\|_{1,p}\|v\|_{1,p}\,.
\end{equation}
\end{theorem}

\begin{proof} Fix $\psi\in C^\infty_c(\Omega)$, and let $\mu_\psi\in {\mathcal M}(\Omega)$ be defined by
\begin{equation}\label{mupsi}
\mu_\psi(B)=\int_{B{\times}\Omega}\psi(y) d\mu(x,y)\hbox{ for all Borel sets $B\subset\subset\Omega$.} 
\end{equation}
 Since 
$$
\int_{\Omega} \varphi(x)d\mu_\psi(x)= \int_{\Omega\times\Omega}\varphi(x)\psi(y)d\mu(x,y) \hbox{ for all $\varphi\in C^\infty_c(\Omega)$,}
$$ 
by the definition of $\|\mu\|_\square$ we have
\begin{equation}\label{tred}
\Bigl|\int_{\Omega} \varphi(x)d\mu_\psi(x)\Bigr|\le
\|\mu\|_\square \|\psi\|_{1,p}\|\varphi\|_{1,p} \hbox{ for all $\varphi\in C^\infty_c(\Omega)$.}
\end{equation}
Hence,
$\mu_\psi\in \mathcal M(\Omega)\cap W^{-1,q}(\Omega)$.

Let $B\subset \Omega$ be a Borel set with $C_{1,p}(B)=0$. Let $\Omega'$ be an open set with $\Omega'\subset\subset\Omega$, and let $\psi_k$ be a non-decreasing sequence of non-negative functions in $C^\infty_c(\Omega')$ converging to the constant $1$ in $\Omega$. 
By the definition of $\mu_{\psi_k}$ and Theorem \ref{BreBro}, we have 
$$
\int_{B\times \Omega} \psi_k(y)d\mu(x,y)=\mu_{\psi_k}(B)=0.
$$
Letting $k\to+\infty$ we deduce $\mu(B\times \Omega')=0$. 
Since this holds for all Borel subsets of $B$ we deduce that $|\mu|(B\times \Omega')=0$. By the arbitrariness of $\Omega'\subset\subset\Omega$ we finally obtain $|\mu|(B\times \Omega)=0$.
For a general $A\subset\Omega$ with $C_{1,p}(A)=0$ it is sufficient to observe that there exists a Borel set $B$ such that $A\subset B\subset \Omega$ and $C_{1,p}(B)=0$. This implies that $A\times\Omega$ and is $|\mu|$-measurable and $|\mu|(A\times\Omega)=0$. A similar argument proves the same result for $\Omega\times A$.

Let $T_{\mu_\psi}\in W^{-1,q}(\Omega)$ be defined as in \eqref{timu} with $\mu$ replaced by $\mu_\psi$.
Fix $u\in W^{1,p}_c(\Omega)\cap L^\infty(\Omega)$. Then, by Theorem \ref{BreBro} and the definition of $\mu_\psi$, we have
$$
\langle T_{\mu_\psi}, u\rangle = \int_\Omega \widetilde u(x) d\mu_\psi(x)=\int_{\Omega\times\Omega} \widetilde u(x)\psi(y) d\mu(x,y).
$$
By \eqref{tred} we have 
$
\|T_{\mu_\psi}\|_{-1,q}\le \|\mu\|_\square \|\psi\|_{1,p}$,
so that the previous equality gives 
\begin{equation}\label{poi}
\Bigl|\int_{\Omega\times\Omega} \widetilde u(x)\psi(y) d\mu(x,y)\Bigr|\le \|\mu\|_\square \|\psi\|_{1,p}\|u\|_{1,p}
\end{equation}
for all $u\in W^{1,p}_c(\Omega)\cap L^\infty(\Omega)$ and every $\psi\in C^\infty_c(\Omega)$.

Fix $u\in W^{1,p}_c(\Omega)\cap L^\infty(\Omega)$ and let $\mu^u\in \mathcal M(\Omega)$ be defined by
\begin{equation}\label{muu}
\mu^u(B):=\int_{\Omega\times B} \widetilde u(x)d\mu(x,y)
\hbox{ for all Borel sets $B\subset\subset\Omega$.} 
\end{equation}
Since
$$
\int_\Omega\psi(y)d \mu^u(y)= \int_{\Omega\times \Omega} \widetilde u(x)\psi(y)d\mu(x,y)\hbox{ for all $\psi\in C^\infty_c(\Omega)$,}
$$
thanks to \eqref{poi} we then have
$$
\Bigl|\int_\Omega\psi(y)d \mu^u(y)\Bigr|\le \|\mu\|_\square \|\psi\|_{1,p}\|u\|_{1,p}\hbox{ for all $\psi\in C^\infty_c(\Omega)$.}
$$
Hence,
$\mu^u\in \mathcal M(\Omega)\cap W^{-1,q}(\Omega)$, and, using the notation above, we obtain
$$
\|T_{\mu^u}\|_{-1,q}\le \|\mu\|_\square \|u\|_{1,p}\,.
$$
By Theorem \ref{BreBro} and the definition of $\mu^u$, we then have
$$
\langle T_{\mu^u}, v\rangle = \int_\Omega \widetilde v(y) d\mu^u(y)=\int_{\Omega\times\Omega} \widetilde u(x)\widetilde v(y) d\mu(x,y).
$$
Together with the previous inequality this gives 
$$
\Bigl|\int_{\Omega\times\Omega} \widetilde u(x)\widetilde v(y) d\mu(x,y)\Bigr|\le
\|\mu\|_\square \|u\|_{1,p}\|v\|_{1,p}\,,
$$
which concludes the proof of \eqref{intcapuv}.\end{proof}

%
%

\section{Continuity properties of some double integrals}\label{continconv}
In this section we find conditions on $f, u_k, v_k, u, v, \mu_k$, and $\mu$ which imply the convergence
$$
\lim_{k\to+\infty} \int_{\Omega\times\Omega} f(\widetilde u_k(x),\widetilde v_k(y))d\mu_k(x,y)=
\int_{\Omega\times\Omega}f(\widetilde u(x),\widetilde v(y))d\mu(x,y).
$$
We begin with the case $f(s,t)=st$. Subsequently, we consider the case when $f$ is a polynomial, and finally an arbitrary continuous function by approximation.

\begin{lemma}\label{lemmaco}
Let $\mu\in\mathcal M^{1,p}_+(\Omega\times\Omega)$ and let $u_k, v_k$ be sequences in $W^{1,p}_c(\Omega)\cap L^\infty(\Omega)$ converging to $u,v$  weakly in $W^{1,p}(\Omega)$.
Assume that there exist a compact set $K\subset\Omega$ and a constant $M$ such that 
\begin{equation}
{\rm supp}(u_k)\cup{\rm supp}(v_k)\subset K\hbox{ and } \|u_k\|_\infty + \|v_k\|_\infty\le M,
\end{equation}
where $\|\cdot\|_\infty$ denotes the norm in $L^\infty(\Omega)$
Then 
\begin{equation}
\lim_{k\to+\infty} \int_{\Omega\times\Omega}\widetilde u_k(x)\widetilde v_k(y)d\mu(x,y)=
\int_{\Omega\times\Omega}\widetilde u(x)\widetilde v(y)d\mu(x,y).
\end{equation}
\end{lemma}

\begin{proof}
We write
\begin{eqnarray}\label{oppo}\nonumber
&&\int_{\Omega\times\Omega}\widetilde u_k(x)\widetilde v_k(y)d\mu(x,y)-
\int_{\Omega\times\Omega}\widetilde u(x)\widetilde v(y)d\mu(x,y)\\
&=&\int_{\Omega\times\Omega}(\widetilde u_k(x)- \widetilde u(x))\widetilde v_k(y)d\mu(x,y)+
\int_{\Omega\times\Omega}\widetilde u(x)(\widetilde v_k(y)-\widetilde v(y))d\mu(x,y).
\end{eqnarray}
The first integral in \eqref{oppo} is estimated by
\begin{eqnarray}\label{supor}\nonumber
\Bigl|\int_{\Omega\times\Omega}(\widetilde u_k(x)- \widetilde u(x))\widetilde v_k(y)d\mu(x,y)
\Bigr|&\le& M\int_{\Omega\times\Omega}|\widetilde u_k(x)- \widetilde u(x)|\psi(y)d\mu(x,y)\\
&=&M \langle T_{\mu_\psi}, |u_k-u| \rangle=o(1)
\end{eqnarray}
as $k\to+\infty$, where $\psi$ is any function in $C^\infty_c(\Omega)$ with $0\le\psi\le 1$ in $\Omega$ and $\psi=1$ on $K$, and $\mu_\psi$ is defined in \eqref{mupsi}, and the convergence to $0$ follows from the fact that $|u_k-u|\rightharpoonup 0$ weakly in $W^{1,p}_0(\Omega)$.
Moreover, if $\mu^u$ is defined in \eqref{muu}, we have 
\begin{equation}\label{kii}
\int_{\Omega\times\Omega}\widetilde u(x)(\widetilde v_k(y)-\widetilde v(y))d\mu(x,y)= \langle T_{\mu^u}, v_k-v\rangle=o(1)
\end{equation}
as $k\to+\infty$. The convergence to $0$ follows from the fact that $v_k-v\rightharpoonup 0$ weakly in $W^{1,p}_0(\Omega)$.
\end{proof}

\begin{corollary}\label{corto}
Under the same hypotheses of Lemma {\rm\ref{lemmaco}}, let
$\mu_k\in\mathcal M^{1,p}(\Omega\times\Omega)$ with $\|\mu_k-\mu\|_\square\to 0$. Then 
\begin{equation}
\lim_{k\to+\infty} \int_{\Omega\times\Omega}\widetilde u_k(x)\widetilde v_k(y)d\mu_k(x,y)=
\int_{\Omega\times\Omega}\widetilde u(x)\widetilde v(y)d\mu(x,y).
\end{equation}
\end{corollary}

\begin{proof}
It suffices to remark that
$$
\Bigl|\int_{\Omega\times\Omega}\widetilde u_k(x)\widetilde v_k(y)d\mu_k(x,y)- \int_{\Omega\times\Omega}\widetilde u_k(x)\widetilde v_k(y)d\mu(x,y)\Bigr|\le \|\mu_k-\mu\|_\square\|u_k\|_{1,p}\|v_k\|_{1,p}
$$
and that the right-hand side tends to $0$. The conclusion then follows from Lemma {\rm\ref{lemmaco}}.
\end{proof}

\begin{proposition}
Let $P\colon \mathbb R^2\to \mathbb R$ be a polynomial function and let $\varphi,\psi\in C^\infty_c(\Omega)$. Let
$\mu_k, \mu\in\mathcal M^{1,p}_+(\Omega\times\Omega)$ with $\|\mu_k-\mu\|_\square\to 0$, let $u_k, v_k$ be sequences in $W^{1,p}_0(\Omega)\cap L^\infty(\Omega)$ converging to $u,v$  weakly in $W^{1,p}_0(\Omega)$ and equibounded in $L^\infty(\Omega)$. Then
\begin{equation}\label{poiu}
\lim_{k\to+\infty} \int_{\Omega\times\Omega}\varphi(x)\psi(y) P(\widetilde u_k(x),\widetilde v_k(y))d\mu_k(x,y)=
\int_{\Omega\times\Omega}\varphi(x)\psi(y)P(\widetilde u(x),\widetilde v(y))d\mu(x,y).
\end{equation}
\end{proposition}

\begin{proof}
It is sufficient to consider $P(s,t)= s^m t^n$ for some non-negative integers $m,n$. We define $w_k(x)= \varphi(x) u_k(x)^m$ and $z_k(y)= \psi(y) v_k(y)^n$. Observe that $w_k, z_k$ satisfy the hypotheses of Lemma \ref{lemmaco}, weakly converging in $W^{1,p}_0(\Omega)$ to $w,z$ given by
$w(x)= \varphi(x) u(x)^m$ and $z(y)= \psi(y) v(y)^n$.
Since
$$
\varphi(x)\psi(y) P(\widetilde u_k(x),\widetilde v_k(y))= 
\widetilde w_k(x)\,\widetilde z_k(y),\qquad
\varphi(x)\psi(y) P(\widetilde u(x),\widetilde v(y))= 
\widetilde w(x)\,\widetilde z(y),
$$
the claim then follows by Corollary \ref{corto}. 
\end{proof}

In the next proposition we consider a stronger condition on the convergence of $\mu_k$ to $\mu$ which allows to avoid the multiplication by the cut-off functions in the previous proposition. Note that the second condition in \eqref{cococo} is necessary for the validity of \eqref{ploi} when $P$ is a constant.

\begin{proposition}\label{contpol}
Let $P\colon \mathbb R^2\to \mathbb R$ be a polynomial function. Let
$\mu_k, \mu\in\mathcal M^{1,p}_+(\Omega\times\Omega)$ with  $\mu_k(\Omega{\times}\Omega)<+\infty$ and  $\mu(\Omega{\times}\Omega)<+\infty$. Suppose that \begin{equation}\label{cococo}\|\mu_k-\mu\|_\square\to 0\quad\hbox{  and }\quad\mu_k(\Omega{\times}\Omega)\to \mu(\Omega{\times}\Omega).\end{equation} Let $u_k, v_k$ be sequences in $W^{1,p}_0(\Omega)\cap L^\infty(\Omega)$ converging to $u,v$  weakly in $W^{1,p}_0(\Omega)$ and with $\sup_k(\|u_k\|_\infty+\|v_k\|_\infty)=M<+\infty$. Then
\begin{equation}\label{ploi}
\lim_{k\to+\infty} \int_{\Omega\times\Omega} P(\widetilde u_k(x),\widetilde v_k(y))d\mu_k(x,y)=
\int_{\Omega\times\Omega}P(\widetilde u(x),\widetilde v(y))d\mu(x,y).
\end{equation}
\end{proposition}

\begin{proof} 
By Remark \ref{rem3}(iii) the first condition in \eqref{cococo} implies that $\mu_k\rightharpoonup\mu$ weakly$^*$ in $\mathcal M(\Omega{\times}\Omega)$. This, together with the second condition in \eqref{cococo}, gives
$\mu_k(B)\to\mu(B)$ for all Borel sets in $\Omega\times\Omega$ such that $\mu(\partial B\cap (\Omega\times\Omega))=0$. Then, for every $\varepsilon>0$ there exists a compact set $K_\varepsilon$ of $\Omega$ such that
\begin{equation}\label{kappaepsilon}
\mu((\Omega\times\Omega)\setminus (K_\varepsilon\times K_\varepsilon))<\varepsilon \ \hbox{ and }\  \mu_k((\Omega\times\Omega)\setminus (K_\varepsilon\times K_\varepsilon))<\varepsilon \hbox{  for every }k.
\end{equation}
Let $\varphi_\e\in C^\infty_c(\Omega)$ with $0\le \varphi_\e\le 1$ in $\Omega$ and $\varphi_\e= 1$ on $K_\varepsilon$, and let
$$
C_M=\max\{ P(s,t) : s,t\in [-M,M]\}.
$$
With this choice 
$$
\Bigl|\int_{\Omega\times\Omega} P(\widetilde u_k(x),\widetilde v_k(y))d\mu_k(x,y)-
\int_{\Omega\times\Omega}\varphi_\e(x) \varphi_\e(y) P(\widetilde u_k(x),\widetilde v_k(y))d\mu_k(x,y)\Bigr|\le C_M \e
$$
and
$$
\Bigl|\int_{\Omega\times\Omega} P(\widetilde u(x),\widetilde v(y))d\mu(x,y)-
\int_{\Omega\times\Omega}\varphi_\e(x) \varphi_\e(y) P(\widetilde u(x),\widetilde v(y))d\mu(x,y)\Bigr|\le C_M \e.
$$
By \eqref{poiu} with $\varphi=\psi=\varphi_\e$ we then deduce 
$$
\limsup_{k\to+\infty}\Bigl|\int_{\Omega\times\Omega} P(\widetilde u_k(x),\widetilde v_k(y))d\mu_k(x,y)-
\int_{\Omega\times\Omega}P(\widetilde u(x),\widetilde v(y))d\mu(x,y)\Bigr|\le 2C_M\e,
$$
from which we obtain \eqref{ploi} by the arbitrariness of $\e$.
\end{proof}

We are now ready to prove the result for an arbitrary continuous function $f$.

\begin{corollary}\label{cortre}
Let $f\colon \mathbb R^2\to \mathbb R$ be a continuous function.
Under the assumptions of Proposition {\rm\ref{contpol}} we have
\begin{equation}\label{ploi2}
\lim_{k\to+\infty} \int_{\Omega\times\Omega} f(\widetilde u_k(x),\widetilde v_k(y))d\mu_k(x,y)=
\int_{\Omega\times\Omega}f(\widetilde u(x),\widetilde v(y))d\mu(x,y).
\end{equation}
\end{corollary}

\begin{proof}
It suffices to approximate uniformly $f$ by polynomials on $[-M,M]^2$ and apply the previous proposition.
\end{proof}

Finally, if $f$ is bounded we can remove the hypothesis that $u_k$ and $v_k$ are bounded in $L^\infty$. 

\begin{theorem}\label{cont-teo}
Let $f\colon \mathbb R^2\to \mathbb R$ be a bounded continuous function. Let
$\mu_k, \mu\in\mathcal M^{1,p}_+(\Omega\times\Omega)$ with  $\mu_k(\Omega{\times}\Omega)<+\infty$ and  $\mu(\Omega{\times}\Omega)<+\infty$. Suppose that \begin{equation}\label{cococo-2}\|\mu_k-\mu\|_\square\to 0\quad\hbox{  and }\quad\mu_k(\Omega{\times}\Omega)\to \mu(\Omega{\times}\Omega).\end{equation} Let $u_k, v_k$ be sequences in $W^{1,p}_0(\Omega)$ converging to $u,v$  weakly in $W^{1,p}_0(\Omega)$, then 
\begin{equation}\label{ploi3}
\lim_{k\to+\infty}\int_{\Omega\times\Omega} f(\widetilde u_k(x),\widetilde v_k(y))d\mu_k(x,y)=
\int_{\Omega\times\Omega}f(\widetilde u(x),\widetilde v(y))d\mu(x,y).
\end{equation}
\end{theorem}

\begin{proof} For all $\lambda>0$ we define the truncation operator 
\begin{equation}\label{trunc}
\tau^\lambda (s):=(s\vee (-\lambda))\wedge\lambda.
\end{equation}
With fixed $\lambda>1$ we set $u^\lambda_k(x):= \tau^{\lambda}(u_k(x))$ and $v^\lambda_k(y):= \tau^{\lambda}(v_k(y))$, and correspondingly, $u^\lambda(x):= \tau^{\lambda}(u(x))$ and $v^\lambda(y):= \tau^{\lambda}(v(y))$.
By the uniqueness  of the quasicontinuous representatives, we deduce from \eqref{muzero} that 
$\widetilde u^\lambda_k(x)= \tau^{\lambda}(\widetilde u_k(x))$ and $\widetilde v^\lambda_k(x)= \tau^{\lambda}(\widetilde v_k(x))$ for $\mu_k$-almost every $(x,y)\in\Omega{\times}\Omega$. Similarly, we have  $\widetilde u^\lambda(x)= \tau^{\lambda}(\widetilde u(x))$ and $\widetilde v^\lambda(x)= \tau^{\lambda}(\widetilde v(x))$ for $\mu$-almost every $(x,y)\in\Omega{\times}\Omega$. 

Since $u^\lambda_k\rightharpoonup u^\lambda$ and  $v^\lambda_k\rightharpoonup v^\lambda$ weakly in $W^{1,p}_0(\Omega)$, we have
\begin{equation}\label{tr1}
\lim_{k\to+\infty} \int_{\Omega\times\Omega} f(\tau^\lambda(\widetilde u_k(x)),\tau^\lambda(\widetilde v_k(y)))d\mu_k(x,y)=
\int_{\Omega\times\Omega}f(\tau^\lambda(\widetilde u(x)),\tau^\lambda(\widetilde v(y)))d\mu(x,y)
\end{equation}
by Corollary \ref{cortre}. To conclude the proof of the result is suffices to estimate
\begin{eqnarray}\label{tr2}
&\displaystyle A^\lambda_k:=\biggl|\int_{\Omega\times\Omega} f(\tau^\lambda(\widetilde u_k(x)),\tau^\lambda(\widetilde v_k(y)))d\mu_k(x,y)-
\int_{\Omega\times\Omega}f(\widetilde u_k(x),\widetilde v_k(y))d\mu_k(x,y)\biggr|,\qquad\ 
\\
\label{tr3}
&\displaystyle
A^\lambda:=\biggl|\int_{\Omega\times\Omega} f(\tau^\lambda(\widetilde u(x)),\tau^\lambda(\widetilde v(y)))d\mu(x,y)-
\int_{\Omega\times\Omega}f(\widetilde u(x),\widetilde v(y))d\mu(x,y)\biggr|.\quad\qquad\qquad\ 
\end{eqnarray}
Let $M_0=\sup |f|$. Since
$$
A^\lambda_k\le 2M_0\,\mu_k(\{(x,y)\in\Omega\times\Omega: (\widetilde u_k(x),\widetilde v_k(y))\not\in [-\lambda,\lambda]^2\}),
$$
it is enough to separately estimate 
\begin{equation}\label{st1}
\mu_k(\{(x,y)\in\Omega\times\Omega: \widetilde u_k(x)> \lambda\}), \qquad \mu_k(\{(x,y)\in\Omega\times\Omega: \widetilde u_k(x)<- \lambda\}),
\end{equation}
\begin{equation}\label{st2}
\mu_k(\{(x,y)\in\Omega\times\Omega: \widetilde v_k(y)> \lambda\}), \qquad \mu_k(\{(x,y)\in\Omega\times\Omega: \widetilde v_k(y)<- \lambda\})\,.
\end{equation}

For fixed $\e>0$ let $K_\e$ be the compact sets introduced at the beginning of the proof of Proposition \ref{contpol}. By \eqref{kappaepsilon} we have
\begin{equation}\label{st3}
\mu_k(\{(x,y)\in\Omega\times\Omega: \widetilde u_k(x)> \lambda\})\le \mu_k(\{(x,y)\in K_\e\times K_\e: \widetilde u_k(x)> \lambda\})+ \e\,.
\end{equation}
Let $\psi_\e\in C^\infty_c(\Omega)$ with $0\le \psi_\e\le 1$ and $\psi_\e=1$ on $K_\e$. 
Let $T_\e$ be the element in $W^{-1,q}(\Omega)$ defined by
$$
\langle T_\e, v\rangle:=\int_{\Omega\times\Omega} \widetilde v(x)\psi_\e(y)d\mu(x,y)
\hbox{ for every }v\in W^{1,p}_0(\Omega).$$
Since $u_k$ are equibounded in $W^{1,p}_0(\Omega)$, by \eqref{intcapuv} there exists $C_\e>0$ and $k_\e\in\mathbb N$ such that
\begin{eqnarray*}
&\displaystyle\mu_k(\{(x,y)\in K_\e\times K_\e: \widetilde u_k(x)> \lambda\})
\le \int_{\Omega\times\Omega} [\widetilde u_k(x)-\lambda+1]^+\psi_\e(x)\psi_\e(y)d\mu_k(x,y)\\
&\displaystyle\le \Bigl|\int_{\Omega\times\Omega}\!\!\! [\widetilde u_k(x)-\lambda+1]^+\psi_\e(x)\psi_\e(y)d(\mu_k-\mu)(x,y)\Bigr|+
\int_{\Omega\times\Omega} \!\!\![\widetilde u_k(x)-\lambda+1]^+\psi_\e(y)d\mu(x,y)\\
&\displaystyle\le C_\e\|\mu_k-\mu\|_\square
+\langle T_{\psi_\e}, [u_k-\lambda+1]^+\rangle
=\e+ \langle T_{\psi_\e}, [u-\lambda+1]^+]\rangle
\end{eqnarray*}
for all $k\ge k_\e$. Now, since 
$
\lim\limits_{\lambda\to+\infty}\langle T_{\psi_\e}, [u-\lambda+1]^+]\rangle=0$,
we can choose $\lambda_\e>0$ such that 
$$
\mu_k(\{(x,y)\in K_\e\times K_\e: \widetilde u_k(x)> \lambda\})\le 2\e 
$$
for all $k\ge k_\e$ and $\lambda\ge\lambda_\e$. By \eqref{st3} this in turn gives 
$$
\mu_k(\{(x,y)\in\Omega\times\Omega: \widetilde u_k(x)> \lambda\})\le 3\e 
$$
for all $k\ge k_\e$ and $\lambda\ge\lambda_\e$.

In the same way, we can prove analogue estimates for the other measures in \eqref{st1} and \eqref{st2}, which we may assume to hold for the same $K_\e$ and $\lambda_\e$, and conclude that
$A^\lambda_k\le 24 M_0\e$
for all $k\ge k_\e$ and $\lambda\ge\lambda_\e$. Similarly we can prove that $A^\lambda\le 24M_0\e$ for $\lambda\ge\lambda_\e$. From these estimates, by \eqref{tr1}--\eqref{tr3} the claim follows by the arbitrariness of~$\e$.
\end{proof}

We finally prove a lower bound for limits of double integrals. 

\begin{corollary}\label{weaksemic}
Let $f\colon \mathbb R^2\to [0,+\infty)$ be a continuous function. Let
$\mu_k, \mu\in\mathcal M^{1,p}_+(\Omega\times\Omega)$ with  $\mu_k(\Omega{\times}\Omega)<+\infty$, $\mu(\Omega{\times}\Omega)<+\infty$, satisfying \eqref{cococo-2}. 
Let $u_k, v_k$ be sequences in $W^{1,p}_0(\Omega)$ converging to $u,v$  weakly in $W^{1,p}_0(\Omega)$. Then
\begin{equation}\label{semicf}
\liminf_{k\to+\infty} \int_{\Omega\times\Omega} f(\widetilde u_k(x),\widetilde v_k(y))d\mu_k(x,y)\ge \int_{\Omega\times\Omega} f(\widetilde u(x),\widetilde v(y))d\mu(x,y).
\end{equation}
\end{corollary}

\begin{proof} For every $\lambda>0$ let $f^\lambda:=\tau^\lambda (f)$, where $\tau^\lambda$ is the truncation operator as in \eqref{trunc}. By Theorem \ref{cont-teo} we then have 
\begin{eqnarray*}
\liminf_{k\to+\infty} \int_{\Omega\times\Omega} f(\widetilde u_k(x),\widetilde v_k(y))d\mu_k(x,y)&\ge& 
\lim_{k\to+\infty} \int_{\Omega\times\Omega} f^\lambda(\widetilde u_k(x),\widetilde v_k(y))d\mu_k(x,y)
\\
&=& \int_{\Omega\times\Omega} f^\lambda(\widetilde u(x),\widetilde v(y))d\mu(x,y).
\end{eqnarray*}
We then conclude by letting $\lambda\to+\infty$.
\end{proof}

\section{$\Gamma$-convergence}\label{gammaconv}
In this final section we shall prove the $\Gamma$-convergence and the Mosco convergence of sequences of functionals as in \eqref{effe-k}.

Let $f\colon \mathbb R^2\to [0,+\infty)$ be a continuous function, and let
$\mu_k, \mu\in\mathcal M^{1,p}_+(\Omega\times\Omega)$. 
We define $F_k,F\colon  W^{1,p}_0(\Omega)\to[0,+\infty]$ by 
\begin{eqnarray}\label{effeacca}
F_k(u):=\int_{\Omega\times\Omega} f(\widetilde u(x),\widetilde u(y))d\mu_k(x,y)\ \hbox{ and }\  F(u):=\int_{\Omega\times\Omega} f(\widetilde u(x),\widetilde u(y))d\mu(x,y).
\end{eqnarray}

We assume that $f$ satisfies the following condition: there exists an unbounded set $\Lambda\subset[0,+\infty)$ and two constants $a,b\ge 0$ such that
\begin{equation}\label{trundec}
f(\tau^\lambda(s), \tau^\lambda(t)) \le a\, f(s,t) + b\quad\hbox{for all $s,t\in\mathbb R$ and for all $\lambda\in\Lambda$,}
\end{equation}
where $\tau^\lambda$ is the truncation operator defined in \eqref{trunc}. Note that this condition is valid if $f$ is bounded (with $a=0$) or when $f$ is decreasing by truncations (with $a=1$ and $b=0$).

Let $g_k, g\colon \Omega\times\mathbb R^d\to \mathbb R$ be Carath\'eodory functions satisfying the growth conditions
\begin{equation}\label{hypg}
c_0|\xi|^p\le g_k(x,\xi)\le c_1|\xi|^p+a(x),\quad c_0|\xi|^p\le g(x,\xi)\le c_1|\xi|^p+a(x)\ \hbox{  for all $(x,\xi)\in \Omega\times\mathbb R^d$,} \end{equation} 
for some constants $c_0, c_1>0$ and some function $a\in L^1(\Omega)$. Let $G_k, G\colon W^{1,p}_0(\Omega)\to [0,+\infty)$ be defined by 
\begin{equation}\label{giacca}
G_k(u):= \int_\Omega g_k(x,\nabla u)dx\quad\hbox{ and }\quad G(u):=\int_\Omega g(x,\nabla u)dx.
\end{equation}

\begin{theorem}\label{teogamma} Let
$\mu_k, \mu\in\mathcal M^{1,p}_+(\Omega\times\Omega)$ with  $\mu_k(\Omega{\times}\Omega)<+\infty$ and  $\mu(\Omega{\times}\Omega)<+\infty$. 
Let $F_k, F$ be defined as in \eqref{effeacca} with $f\colon \mathbb R^2\to [0,+\infty)$ a continuous function satisfying \eqref{trundec}. 
Let $G_k, G$ be defined as in \eqref{giacca} with $g_k, g\colon \Omega\times\mathbb R^d\to \mathbb R$  Carath\'eodory functions satisfying \eqref{hypg}. Suppose that
\begin{eqnarray}\label{mug}
&\|\mu_k-\mu\|_\square\to 0\quad\hbox{  and } 
\quad\mu_k(\Omega{\times}\Omega)\to \mu(\Omega{\times}\Omega),\\
\label{gammag}
 &\displaystyle G=\Gamma\hbox{-}\lim_{k\to+\infty} G_k \hbox{ with respect to the weak convergence in $W^{1,p}_0(\Omega)$.}
 \end{eqnarray} 
Then 
$ F+G=\Gamma\hbox{-}\lim\limits_{k\to+\infty} (F_k+G_k)$  with respect to the weak convergence in $W^{1,p}_0(\Omega)$.
\end{theorem}

\begin{proof} By \cite[Proposition 8.10]{DM} we have to prove that for all $u\in W^{1,p}_0(\Omega)$ the following properties hold:

(i)  for all $u_k$ converging to $u$ weakly in $W^{1,p}_0(\Omega)$ we have 
$$
F(u)+G(u)\le \liminf_{k\to+\infty}  (F_k(u_k)+G_k(u_k));$$

(ii) there exist $u_k$ converging to $u$ weakly in $W^{1,p}_0(\Omega)$ such that 
$$
F(u)+G(u)= \lim_{k\to+\infty}  (F_k(u_k)+G_k(u_k)).
$$

Claim (i) follows from Corollary \ref{weaksemic} with $v_k=u_k$ and from the liminf inequality for~$G_k$, which follows from \eqref{gammag}.

If $u\in W^{1,p}_0(\Omega)\cap L^\infty(\Omega)$, we deduce from  \eqref{gammag} that there exists a sequence $u_k$ converging to $u$ weakly in $W^{1,p}_0(\Omega)$ such that $G(u_k)\to G(u)$ and $\|u_k\|_\infty\le M$ for some constant $M$ and for all $k$ (see for instance \cite[Proposition 2.5]{Butt-DM}). Setting
$
\lambda=\max \{f(s,t): |s|\le M, |t|\le M\}$,
we may apply Theorem \ref{cont-teo} with $v_k=u_k$ and $\tau^\lambda (f)$ in the place of $f$, obtaining (ii).

Let now $u\in W^{1,p}_0(\Omega)$. If $F(u)=+\infty$, then claim (ii) follows from claim (i). Suppose then that $F(u)<+\infty$. By the validity of claim (ii) for $\tau^\lambda(u)$ we obtain
\begin{equation}\label{gammalimsupf}
F(\tau^\lambda(u))+G(\tau^\lambda(u))=\Bigl(\Gamma\hbox{-}\limsup_{k\to+\infty} (F_k+G_k)\Bigr)(\tau^\lambda(u)).
\end{equation}
By \eqref{hypg} we have $G(\tau^\lambda(u))\to G(u)$ as $\lambda\to+\infty$. Since $F(u)<+\infty$, by \eqref{trundec} and the Dominated Convergence Theorem we have $F(\tau^\lambda(u))\to F(u)$ as $\lambda\to+\infty$ with $\lambda\in\Lambda$. 
By \eqref{gammalimsupf}, using the lower semicontinuity of the $\Gamma$-limsup we get 
$$
F(u)+G(u)\ge\Bigl(\Gamma\hbox{-}\limsup_{k\to+\infty} (F_k+G_k)\Bigr)(u).
$$
By \cite[Proposition 8.10]{DM} we then obtain an inequality in claim (ii). The proof is completed by using claim (i).
\end{proof}

\begin{remark}[Mosco convergence]\rm
If the functionals $G_k$ converge to $G$ in the sense of the Mosco convergence in $W^{1,p}_0(\Omega)$; that is, for all $u\in W^{1,p}_0(\Omega)$ we have

(i)  for all $u_k$ converging to $u$ weakly in $W^{1,p}_0(\Omega)$ we have 
$$
G(u)\le \liminf_{k\to+\infty}  G_k(u_k);$$

(ii) there exist $u_k$ converging to $u$ strongly in $W^{1,p}_0(\Omega)$ such that 
$$
G(u)= \lim_{k\to+\infty}  G_k(u_k),
$$
then also $F_k+G_k$ converges in the sense of the Mosco convergence in $W^{1,p}_0(\Omega)$.
%
\end{remark}

%
%

\noindent {\bf Acknowledgements.}
 This paper is based on work supported by the National Research Project (PRIN  2017) 
 ``Variational Methods for Stationary and Evolution Problems with Singularities and 
 Interfaces", funded by the Italian Ministry of University, and Research. 
The authors are members of the Gruppo Nazionale per 
l'Analisi Matematica, la Probabilit\`a e le loro Applicazioni (GNAMPA) of the 
Istituto Nazionale di Alta Matematica (INdAM).


\begin{thebibliography}{12}

 
\bibitem{B}A. Braides: {\em $\Gamma$-convergence for Beginners}. Oxford University Press, Oxford, 2002.

\bibitem{BCD} A. Braides, P. Cermelli, S. Dovetta:
$\Gamma$-limit of the cut functional on dense graph sequences.
{\it ESAIM Control Optim. Calc. Var.} {\bf 26} (2020), 26.

\bibitem{BB} H. Brezis, F. Browder: A property of Sobolev spaces. {\em Comm. Partial Differential Equations} {\bf 4} (1979), 1077--1083. 

\bibitem{Butt-DM} G. Buttazzo, G. Dal Maso:  $\Gamma$-limits of integral functionals. {\em J. Analyse Math.} {\bf 37} (1980), 145--185.

\bibitem{Bo} F.R Chung, R.L. Graham, R.M. Wilson:  Quasi-random graphs. {\em Combinatorica} {\bf 9} (1989),  345--362.
 
\bibitem{DM} G. Dal Maso:
{\em An Introduction to $\Gamma$-convergence}. Birkh\"auser, Boston, 1993.

\bibitem{EG} L.C. Evans, R.F. Gariepy: {\em Measure Theory and Fine Properties of Functions. Revised edition.} Textbooks in Mathematics. CRC Press, Boca Raton, FL, 2015. 

\bibitem{Fukushima}  M. Fukushima, Y. Oshima, M. Takeda: {\it Dirichlet Forms and Symmetric Markov Processes. Second revised and extended edition.\/} De Gruyter Studies in Mathematics, 19. Walter de Gruyter \& Co., Berlin, 2011.

\bibitem{Grun-Rehomme}  M. Grun-Rehomme: Caract\'erisation du sous-diff\'erentiel d'int\'egrandes convexes dans les espaces de Sobolev.  {\it J. Math. Pures Appl.} {\bf 56} (1977), 149–156. 

\bibitem{finlandesi} J. Heinonen, T. Kilpeläinen, O. Martio: {\it Nonlinear Potential Theory of Degenerate Elliptic Equations. Unabridged republication of the 1993 original.\/} Dover Publications Inc., Mineola, NY, 2006.

\bibitem{Lo} L. Lov\'asz: {\em Large Networks and Graph Limits}.  American Mathematical Society Publications, Providence, RI, 2012.


\bibitem{Mosco-2}U. Mosco:
Convergence of convex sets and of solutions of variational inequalities. 
{\it Advances in Math.\/} {\bf 3} (1969), 510--585. 

\bibitem{Mosco} U. Mosco:
Composite media and asymptotic Dirichlet forms,
{\em J. Funct. Anal.} {\bf 123} (1994), 368--421,

\end{thebibliography}
\end{document}